%% file: main.tex
\documentclass[letterpaper, 10 pt, conference]{ieeeconf}  % Comment this line out
\IEEEoverridecommandlockouts                              % This command is only
\usepackage{graphics} % for pdf, bitmapped graphics files
\usepackage{epsfig} % for postscript graphics files
\usepackage{mathptmx} % assumes new font selection scheme installed
\usepackage{times} % assumes new font selection scheme installed
\usepackage{amsmath,amsfonts} % assumes amsmath package installed
\usepackage{amssymb}  % assumes amsmath package installed
\usepackage{tabularx}
\usepackage{algorithm}
\usepackage{algorithmic}
\usepackage{subcaption}
\usepackage{url}
\usepackage[normalem]{ulem}
\usepackage{dsfont}
\usepackage{graphicx}
\usepackage[dvipsnames,usenames]{xcolor}
\usepackage[colorlinks,%
linkcolor=BrickRed,%
filecolor=BrickRed,%
citecolor=RoyalPurple,%
]{hyperref}
\usepackage{float}

\DeclareFontEncoding{LS1}{}{}
\DeclareFontSubstitution{LS1}{stix}{m}{n}
\DeclareSymbolFont{stixletters}{LS1}{stix}{m}{it}
\DeclareMathAccent{\cev}{\mathord}{stixletters}{"91}

\input{_definitions}

\newcommand{\F}{{\mathcal F}}
\newcommand{\Xr}{\tilde X}
\newcommand{\Xo}{X^o}
\newcommand{\fr}{\tilde f}
\newcommand{\gr}{\tilde g}
\newcommand{\Ur}{\tilde U}
\newcommand{\Wr}{\tilde W}
\newcommand{\Yr}{\tilde Y}
\newcommand{\Zr}{\tilde Z}
\newcommand{\Yn}{\mathsf Y}

\title{\LARGE \bf
Time-Reversed BSDEs for Accurate Gradient Estimation in\\ Diffusion Models
}

\author{
Yuhang Mei$^\star$ and
Amirhossein Taghvaei$^\star$
  {\thanks{$^\star$Department of Aeronautics \& Astronautics, University of Washington, Seattle; {\tt\small yuhangm@uw.edu ,amirtag@uw.edu}. This research is supported by the National Science Foundation (NSF) award EPCN-2347358.}}
}

\begin{document}

\maketitle
\thispagestyle{empty}
\pagestyle{empty}

%%%%%%%%%%%%%%%%%%%%%%%%%%%%%%%%%%%%%%%%%%%%%%%%%%%%%%%%%%%%%%%%%%%%%%%%%%%%%%%%
\begin{abstract}
There is a growing literature adopting a stochastic optimal control (SOC) perspective to fine-tune diffusion models and related generative policies. A prominent class of methods, known as \emph{iterative diffusion optimization}, solves the SOC problem by simulating the diffusion process, evaluating a loss function, and applying stochastic optimization algorithms, with \emph{adjoint matching} emerging as a state-of-the-art approach. However, the adjoint process used in these methods is  not adapted to the forward diffusion filtration, which can lead to unstable or high-variance gradient estimates.
In this paper, we revisit gradient estimation in diffusion models through the lens of backward stochastic differential equations (BSDEs). We propose an alternative estimator based on a time-reversed BSDE formulation introduced in our prior work, which produces an adjoint process adapted to the underlying filtration. This adapted structure leads to more stable gradient estimates with potentially lower variance.
We analyze the accuracy of the proposed estimator and compare it with adjoint matching. Numerical experiments on fine-tuning toy diffusion models demonstrate improved gradient stability and competitive performance.
\end{abstract}

%%%%%%%%%%%%%%%%%%%%%%%%%%%%%%%%%%%%%%%%%%%%%%%%%%%%%%%%%%%%%%%%%%%%%%%%%%%%%%%%
\section{Introduction}
\label{sec:introduction}
Generative diffusion models~\cite{song2019generative,ho2020denoising,songscore2021score} have demonstrated remarkable capability in generating samples from complex high-dimensional distributions, across a wide range of applications, including images synthesis~\cite{rombach2022high},
protein structure generation~\cite{watson2023novo}, 
text generation~\cite{li2022diffusion}, 
video generation~\cite{ho2022video},
and robotics control policy generation~\cite{chi2025diffusion}. In many practical scenarios, however, it is desirable to steer the generation process toward outputs that satisfy additional objectives, preferences, or task-specific requirements. Compared to training a new diffusion model from scratch, fine-tuning an existing model provides a  more efficient and cost-effective approach. 

A common strategy to fine-tuning diffusion models is to introduce a reward function that measures the quality of generated samples. This formulation naturally enables the use of both reinforcement learning (RL)~\cite{fan2023dpok,blacktraining} and stochastic optimal control (SOC) methods~\cite{uehara2024fine,domingoadjoint,tang2024fine}.
The interplay between SOC and generative modeling has recently garnered significant attention~\cite{zhang2022fast,rapakoulias2024go,chen2022likelihood,elamvazhuthi2025flow,mei2025synthesis,Mei2025FlowMatching}.
% Empirical results reported in~\cite[Tables 1--3]{uehara2024fine} suggest that SOC-based methods can outperform RL-based approaches.
Motivated by these developments, we adopt the SOC perspective in this paper and focus on a class of SOC-based methods known as \emph{iterative diffusion optimization}\cite{nusken2021solving}. These methods fine-tune the model by repeatedly simulating the diffusion process, evaluating a loss function on the generated samples, and updating the parameters via stochastic optimization. Among them, adjoint matching has emerged as a state-of-the-art technique~\cite{domingoadjoint}. However, the adjoint process employed in these methods is  not adapted to the forward diffusion filtration, which may lead to unstable or high-variance gradient estimates.

These limitations motivate the development of more principled adjoint methods in which the adjoint process is adapted to the forward filtration and characterized by a backward stochastic differential equation (BSDE), as introduced in the stochastic maximum principle~\cite{yong1999stochastic}. In our prior work~\cite{taghvaei2024time,mei2025time}, we introduced a novel time-reversal methodology for solving such BSDEs and showed that it compares favorably with existing BSDE solvers in the linear setting. In this paper, we extend this numerical framework to the nonlinear setting and demonstrate its effectiveness for fine-tuning nonlinear diffusion models.

Here are the contributions of the paper: (i) We propose a novel and accurate gradient estimator for fine-tuning diffusion models; (ii) we present both adapted and non-adapted adjoint processes and establish a theoretical comparison; (iii) We conduct extensive numerical experiments comparing our method with baseline approaches on both linear and nonlinear stochastic models, and demonstrate its effectiveness, outperforming adjoint matching in a full fine-tuning task on a toy diffusion model.

% This motivates the development of more principled gradient estimation methods that leverage both the structure of diffusion models and the theoretical framework of SOC. In this paper, we propose a novel and accurate gradient estimator derived from BSDE. It is based on the time-reversed BSDEs methodology proposed in our prior work~\cite{taghvaei2024time,mei2025time} to backpropagate and ``denoise'' the gradient.

% We present our algorithm and further discuss the relationship between the adapted adjoint process and the non-adapted adjoint process. The numerical implementation of our algorithm consists of four main steps: 1. forward simulation of the SDE; 2. time-reversal simulation of the forward SDE; 3. time-reversal simulation of the backward SDE; 4. learning a time-varying function via regression between the trajectories of the time-reversed forward SDE and the time-reversed BSDE.
% The numerical experiments include a comparison of gradient estimation on both linear and nonlinear models. In addition, we implement the full fine-tuning procedure on a toy diffusion model and compare the performance of our method with adjoint matching. The results demonstrate improved gradient estimation and competitive performance on diffusion model fine-tuning tasks.

The paper is organized as follows: (i) the SOC problem for fine-tuning diffusion models is introduced in Sec.~\ref{sec:prob}; (ii) the non-adapted and adapted adjoint processes are introduced and compared in Sec.~\ref{sec:adjoint}; (iii) the proposed time-reveresal methodology to solve the BSDE is presented in Sec.~\ref{sec:method}; (iv) finally, numerical results comparing the performance of our method with adjoint matching are presented in Sec.~\ref{sec:num}.

%%%%%%%%%%%%%%%%%%%%%%%%%%%%%%%%%%%%%%%%%%%%%%%%%%%%%%%%%%%%%%%%%%%%%%%%%%%%%%%%
\section{Problem Formulation}\label{sec:prob}
Consider a $n$-dimensional diffusion process $\Xo=\{\Xo_t;\,0\leq t\leq T\}$ governed by the following stochastic differential equation (SDE)
\begin{equation}\label{eq:fsde}
    \ud \Xo_t = f(t,\Xo_t) \ud t + g(t,\Xo_t) \ud W_t, ~~\Xo_0 \sim p_0,
\end{equation}
where $W=\{W_t;\,0\leq t\leq T\}$ is the $n$-dimensional standard Wiener process, $p_0$ is the initial distribution, $f: \R \times \R^n \to \R^n$ is the drift function, and $g: \R \times \R^n \to \R^{n\times n}$ is the diffusion matrix.  These functions are assumed to satisfy the usual regularity conditions that ensure the existence of a strong solution~\cite[Thm. 5.2.1]{oksendal2003stochastic}. 
We interpret the SDE~\eqref{eq:fsde} as a pre-trained diffusion model~\cite{songscore2021score}. Specifically, the initial distribution, the drift function, and the diffusion matrix are designed such that the distribution of $\Xo_T$, at the terminal time $T$, matches a target distribution $p_T$ representing data. 
% That is to say, the diffusion model generates samples $X_T$ from the target distribution $p_T$. 
%While we keep the formulation of the SDE~\eqref{eq:fsde} general, we note that in many practical diffusion models, the initial distribution $p_0$ is Gaussian and the diffusion matrix $g(t,x)$ does not depend on the state $x$. 

We are interested in the problem of fine-tuning the diffusion model so that it generates samples  from the ``tilted" target distribution 
\begin{align*}
    q^\star_T \propto p_Te^{-\ell_f},
\end{align*}
where $\ell_f$ is a user-specified loss function that evaluates the quality of generated samples.
% and measures how well they align with the objectives of the fine-tuning task.
%When $\beta=0$, the tilted distribution is equal to $p_T$, corresponding to no modification to the original diffusion model. As $\beta \to \infty$, the distribution increasingly concentrates on regions where $\ell_f$ is minimized. The function $\ell_f$ encodes desirable outcomes by serving as a terminal cost (or reward) that evaluates the quality of generated samples and measures how well they align with the objectives of the fine-tuning task.
The modified or the controlled diffusion is denoted by $X=\{X_t;\,0\leq t\leq T\}$ and modeled by 
%the SDE
\begin{equation}\label{eq:fsde-modified}
    \ud X_t = f(t,X_t) \ud t + g(t,X_t) (U_t\ud t + \ud W_t), ~~X_0 \sim q_0,
\end{equation}
where $q_0$ is the modified initial distribution, and $U=\{U_t;\,0\leq t\leq T\}$ is the control input process which is assumed to be adapted to the filtration generated by the initial condition and the Wiener process.  

The problem of sampling from the tilted target distribution can now be formulated as a stochastic control problem in which the initial distribution $q_0$ and the control input $U$ are selected so that the distribution of $X_T$ matches the tilted target distribution $q^\star_T$. Using a KL-regularized variational characterization of the tilted distribution, together with Girsanov’s theorem, this problem can be expressed as the following stochastic optimal control problem:
\begin{align}\label{eq:SOC}
    \min_{q_0,U}\, \Expect\left[\ell_f(X_T) + \int_0^T \frac{1}{2} \|U_t\|^2 \ud t \right] + D_\mathrm{KL}(q_0\|p_0).
\end{align}

Stochastic optimal control problems of this form have a long history in the control literature, particularly in path-integral control~\cite{kappen2005linear,todorov2006linearly,theodorou2010generalized}, where the control-affine structure of the model~\eqref{eq:fsde-modified} and the quadratic control cost in~\eqref{eq:SOC} arises naturally through a change-of-measure argument related to Girsanov’s theorem. 
% Indeed, letting $P_{\tilde X}$ and $P_{X}$ denote the path probability law for the original diffusion process $\tilde X$ and the  controlled diffusion process $X$, respectively, the SOC problem~\eqref{eq:SOC} may be expressed as
% \begin{align}\label{eq:SOC-KL}
%     \min_{q_0, U}\, \Expect[\beta\ell_f(X_T)] + D_\mathrm{KL}(P_{ X} \| P_{\tilde X}).
% \end{align}
In particular, the optimal 
controlled diffusion process that solves the SOC problem~\eqref{eq:SOC} has the path probability law
\begin{align*}
    \frac{\ud P_{X}}{\ud P_{\Xo}} = \frac{\exp\left(- \ell_f(\Xo_T)\right)}{\Expect[\exp\left(- \ell_f(\Xo_T)\right)]},
\end{align*}
where the expectation is taken with respect to the original diffusion process~\eqref{eq:fsde}. Specifically, the terminal distribution of the optimal controlled process $X_T$ coincides with the tilted target distribution $q^\star_T$. The  path probability law allows us to derive expressions for the optimal initial distribution and the optimal control~\cite[Thm. 1]{uehara2024fine}, however, accurately estimating the expectations is challenging due to the high variance of the exponential weight $\exp(-\ell_f(\Xo_T))$~\cite{williams2017model, exarchos2018stochastic}. 
%A common remedy is to employ importance sampling by introducing an auxiliary control input to modify the diffusion dynamics when evaluating these expectations~\cite{}. 
In this paper, we investigate an alternative class of optimization-based approaches, motivated by recent advances in diffusion models. 

\section{Adjoint processes}\label{sec:adjoint}
In this section, we present two adjoint processes that arise in the gradient estimation of the stochastic optimal control problem~\eqref{eq:SOC}.  
The first is a \emph{non-adapted adjoint process} that appears in adjoint matching methods and is obtained via pathwise differentiation of the objective.  
The second is an \emph{adapted adjoint process} obtained from a BSDE.  

\subsection{Notations} 
Let $(\Omega,\mathcal F, \{\mathcal F_t\}_{t\geq 0}, \mathbb P)$ be a filtered probability space and assume the Wiener process $W_t$ is adapted to the filteration $\mathcal F_t$. We will use the following Hilbert spaces throughout the exposition: (i) The  Euclidean space $\Re^n$ equipped with the inner-product $\langle x,y\rangle:=x^\top y$; (ii) The space of deterministic control trajectories $L^2([0,T];\mathbb R^n)$ equipped with the  inner product $\langle u,v \rangle := \int_0^T u_t^\top v_t \ud t$; (iii) The space of $\Re^n$-valued $\mathcal F_0$-measurable random variables $L^2_{\mathcal F_0}(\Omega;\Re^n)$ equipped with the inner-product $\langle X,Y \rangle := \Expect[X^\top Y]$; (iv) The space of $\mathcal F_t$-adapted  stochastic control processes $L^2_\mathcal F(\Omega \times [0,T];\Re^n)$ equipped with the inner product $\langle U,V\rangle :=\Expect[\int_0^T U_t^\top V_t \ud t] $. 

% \begin{itemize}
% \item The  Euclidean space $\Re^n$ equipped with the inner-product $\langle x,y\rangle:=x^\top y$.
%     \item  The space of deterministic control trajectories $L^2([0,T];\mathbb R^n)$ equipped with the  inner product $\langle u,v \rangle := \int_0^T u_t^\top v_t \ud t$.
%     \item The space of $\Re^n$-valued $\mathcal F_0$-measurable random variables $L^2_{\mathcal F_0}(\Omega;\Re^n)$ equipped with the inner-product $\langle X,Y \rangle := \Expect[X^\top Y]$.
%     \item The space of $\mathcal F_t$-adapted  stochastic control processes $L^2_\mathcal F(\Omega \times [0,T];\Re^n)$ equipped with the inner product $\langle U,V\rangle :=\Expect[\int_0^T U_t^\top V_t \ud t] $. 
% \end{itemize}
 For a functional $F:\mathcal H\to\mathbb R$ defined on a Hilbert space $\mathcal H$ with inner product $\langle\cdot,\cdot\rangle_{\mathcal H}$, the G\^ateaux derivative of $F$ at $u\in \mathcal H$ is denoted by  $\frac{\partial F}{\partial u}(u) \in \mathcal H$ and satisfies
\begin{align*}
    \frac{\ud}{\ud t}(F(u+tv)) = \langle \frac{\partial F}{\partial u}(u), v \rangle_{\mathcal H},\quad \forall v \in \mathcal H.  
\end{align*}

For any initial condition $x\in \mathbb R^n$ and deterministic control $u \in L^2([0,T];\Re^n)$, we use $X^{x,u}_t$ to denote the solution to the controlled SDE~\eqref{eq:fsde-modified} with $X^{x,u}_0=x$ and  $U_t=u_t$. Similarly, for any  $\mathcal F_0$-measurable random variable $\xi \in L^2_{\F_0}(\Omega;\Re^n)$, and  admissible stochastic control process  $U \in L^2_{\mathcal F}([0,T];\Re^n)$,  we denote by $X^{\xi,U}_t$ the solution to the same SDE with the random initial condition $X^{\xi,U}_0=\xi$ and the control input $U$. In particular, the controlled diffusion~\eqref{eq:fsde-modified} can be written as $X_t=X^{\xi,U}_t$ where the distribution of $\xi$ is equal to $q_0$. 

Using this notation, define  the sample objective function 
\begin{align*}
 &J(x,u)(\omega) :=  \beta \ell_f(X^{x,u}_T(\omega)) + \int_0^T \frac{1}{2} \|u_t\|^2 \ud t, \quad \omega \in \Omega.
\end{align*}
When the initial condition and control are random, the corresponding objective function is written as
\begin{align*}
 &J(\xi,U) =  \beta \ell_f(X^{\xi,U}_T) + \int_0^T \frac{1}{2} \|U_t\|^2 \ud t. 
\end{align*}
With this notation, the SOC objective function in~\eqref{eq:SOC} can be expressed compactly as
\begin{align*}
\Expect[ J(\xi,U)]  +  D_\mathrm{KL}(q_0\|p_0),\quad \text{where}\quad \xi \sim q_0.
\end{align*}

For the purposes of this section we temporarily ignore the KL-divergence term and focus on the derivatives of $\mathbb E[J(\xi,U)]$.  The derivative of the KL term will be addressed separately in Sec.~\ref{sec:num-diff}.

\subsection{Non-adapted adjoint process}
Throughout this subsection we assume the diffusion matrix is independent of the state, i.e.
$
g(t,x)=g(t)
$.
% Adjoint matching methods compute gradients by introducing a pathwise adjoint process.  
The non-adapted adjoint process arises when computing the derivatives of the sample objective  $J(x,u)(\omega)$ with respect to the initial condition $x \in \Re^n$ and deterministic control trajectory $u \in L^2([0,T];\Re^n)$, for a fixed sample $\omega \in \Omega$. Under suitable regularity conditions~\cite[Ch. V, Thm. 39]{protter2005stochastic}, the mapping $x \to X^{x,u}_t(\omega)$ is known to be differentiable  and the sensitivity matrix $ \eta^x_t(\omega):=\frac{\partial X^{x,u}_t(\omega)}{\partial x} \in \Re^{n\times n}$ solves the differential equation
    \begin{align*}
        \frac{\ud}{\ud t} \eta^x_t = \frac{\partial f}{\partial x}(t,X^{x,u}_t) \eta^{x}_t ,\quad \eta^x_0 =  I. 
    \end{align*}
The (non-adapted) adjoint process $\Yn^{x,u}_t$, is defined by the backward differential equation
\begin{equation}\label{eq:non-ad}
\frac{\ud \Yn^{x,u}_t}{\ud t} =  - [\frac{\partial f}{\partial x}(t,X^{x,u}_t)]^\top \Yn^{x,u}_t ,\quad \Yn^{x,u}_T = \nabla \ell_f(X^{x,u}_T). 
\end{equation}
so that the duality relationship
    $
        \frac{\ud}{\ud t} \big( (\Yn^{x,y}_t)^\top \eta^x_t \big) = 0 
    $
    holds true. 
 This identity allows the derivatives of the objective to be expressed in terms of the adjoint process:
     \begin{align*}
\frac{\partial J}{\partial x} (x,u)(\omega) &= \Yn^{x,u}_0(\omega),\\
\left(\frac{\partial J}{\partial u} (x,u)(\omega)\right)_t &= u_t + g(t)^\top  \Yn^{x,u}_t(\omega).
\end{align*}
The derivation follows from a standard variational analysis of the objective function $J(x,u)(\omega)$, for a fixed $\omega \in \Omega$, combined with the above duality relation and is omitted for brevity. 

 For a random initial condition $\xi \in L^2_{\mathcal F_0}(\Omega;\Re^n)$ and a stochastic control $U \in L^2_{\mathcal F}(\Omega \times [0,T];\Re^n)$, the corresponding adjoint process is denoted by $\Yn_t$ and formally defiend as 
\[\Yn_t(\omega):=\Yn^{\xi(\omega),U(\omega)}_t(\omega).\]

\subsection{Adapted adjoint process}
The adjoint process appearing in the previous subsection is not $\mathcal F_t$-adapted.  
In contrast, the stochastic maximum principle (SMP) introduces an adapted adjoint process obtained as the solution of a BSDE~\cite{yong1999stochastic}. 
Define the Hamiltonian
\begin{equation}
    H(t, x, y, z,u) := \frac{1}{2}\|u\|^2 + y^\top (f(t,x) + g(t,x)u) +  \trace(z g(t,x)),
\end{equation}
for $(t,x,y,z,u) \in [0,T]\times \R^n \times \R^n \times \R^{n\times n} \times \R^n$. The adapted adjoint process $(Y_t,Z_t)$ is defined by the BSDE
\begin{equation}
-\ud Y_t
=
\frac{\partial H}{\partial x}(t,X_t,Y_t,Z_t,U_t)\,\ud t
-
Z_t\,\ud W_t,~
Y_T
=
\nabla \ell_f(X_T).
\label{eq:bsde}
\end{equation}
% Together with the state dynamics~\eqref{eq:fsde-modified}, this forms a forward–backward stochastic differential equation (FBSDE). 
The adaptedness of $(Y_t,Z_t)$ enables a variational analysis of $\Expect[J(\xi,U)]$~\cite[Sec. 4.3]{yong1999stochastic}. In particular, for any perturbations $\zeta \in L^2_{\F_0}(\Omega;\Re^n
)$ and $V \in L^2_{\F}(\Omega\times [0,T];\Re^n
)$, we have
\begin{align*}
\frac{\ud}{\ud \epsilon}
\mathbb E[J(\xi+\epsilon\zeta,U)]
\Big|_{\epsilon=0}
&=
\mathbb E[Y_0^\top\zeta] \\ 
\frac{\ud}{\ud \epsilon}
\mathbb E[J(\xi,U+\epsilon V)]
\Big|_{\epsilon=0}  &= 
\mathbb E\!\left[
\int_0^T
\frac{\partial H}{\partial u}
(t,X_t,Y_t,Z_t,U_t)^\top V_t
\,\ud t
\right].
\end{align*} 
Consequently, following the definition of G\^ateaux derivative, 
\begin{align*}
\frac{\partial}{\partial \xi}\mathbb E[J(\xi,U)]
&=
Y_0,
\\
\left(\frac{\partial}{\partial U}\mathbb E[J(\xi,U)]\right)_t
&=
U_t + g(t,X_t)^\top Y_t .
\end{align*}
This variational analysis differs slightly, and is in fact simpler, than the one used in the SMP. In particular, it considers variations in the Hilbert space $L^2_{\mathcal F}(\Omega \times [0,T];\Re^n)$ rather than the needle-type perturbations that are typically employed in the SMP framework. 
% The adaptedness of the adjoint process is the key difference and unique in allowing us to derive an expression for the difference between the objective function evaluated at two different operating points:
% \begin{align*}
%    \frac{\ud}{\ud \epsilon}&\left[\Expect[J(\xi+\epsilon \zeta,U+\epsilon V)- J(\xi,U)]\right]_{\epsilon=0} = \Expect[Y_0^\top \zeta]  \\  &+\Expect\left[\int_0^T \frac{\partial H}{\partial u} (t, X_t,U_t,Y_t,Z_t) \ud t \right]
% \end{align*}
% concluding the derivatives
% \begin{align*}
%     \frac{\partial}{\partial \xi}\Expect[ J(\xi,U)] &= Y_0\\
%  \left(\frac{\partial}{\partial U}\Expect[ J(\xi,U)]\right)_{t}&=   U_t + g(t,X_t)^\top Y_t 
% \end{align*}
% where we used the fact that $(X_t,U_t,Y_t,Z_t)$ is $\mathcal F_t$-measurable.  

% \begin{remark}\label{rm:fbsde-hardness}
%     Notice that for BSDE the filtration $\mathcal F_t$ is generated forward in time but the constraint is at the terminal time $T$, so solving FBSDEs is both numerically and conceptually hard. In our previous work, we proposed a method of leveraging time-reversal of diffusion to solving the FBSDEs and in next section we will review the time-reversal of diffusion.
% \end{remark}

\subsection{Relationship between adapted and non-adapted adjoint processes}

Assume again that $g(t,x)=g(t)$. A relationship between the two adjoint processes is provided using the linearity of the BSDE~\eqref{eq:bsde}.  Let
$
A_t
:=
\frac{\partial f}{\partial x}(t,X_t), 
$
and define the matrix processes
\begin{align*}
\ud \Phi_t &= -A_t^\top \Phi_t\,\ud t,
\quad 
\ud \Psi_t = \Psi_t A_t^\top\,\ud t,
\qquad
\Phi_0=\Psi_0 = I .
\end{align*}
By the application of the It\^o rule, 
one can verify that
$
 \Psi_t \Phi_t = I
$, or $\Phi_t = \Psi_t^{-1}$, for all $t\geq 0$. These matrix-valued processes are used to provide an explicit solution to the BSDE~\cite[Sec. 7.2]{yong1999stochastic}. 
The BSDE dynamics yields 
$
\ud(\Psi_t Y_t) = \Psi_t Z_t\,\ud W_t 
$. 
Integrating from $t$ to $T$, and multiplying by $\Phi_t$, gives \[
Y_t
=
\Phi_t\Psi_T \nabla \ell_f(X_T)
-
\int_t^T \Phi_t\Psi_s Z_s\,\ud W_s, 
\]
where we used the terminal condition $Y_t = \nabla \ell_f(X_T)$. 
% Multiplying by $\Phi_t$ and 
The solution $Y_t$ is obtained by taking conditional expectation
\[
Y_t
=
\mathbb E
\left[\Phi_t 
\Psi_T \nabla \ell_f(X_T)
\mid
\mathcal F_t
\right].
\]
On the other hand, the non-adapted adjoint satisfies
$
\Yn_t
=
\Phi_t \Psi_T
\nabla \ell_f(X_T)
$, 
concluding the identity 
\[
Y_t
=
\mathbb E[\Yn_t\mid\mathcal F_t].
\]
Thus the adapted adjoint process can be interpreted as the conditional expectation of the pathwise adjoint process, implying variance reduction  by projection onto the filtration.
Moreover, the difference between $Y_t$ and $\Yn_t$ can be written explicitly as 
\[
\Yn_t - Y_t
=
\int_t^T
\Phi_t\Psi_s Z_s\,\ud W_s. 
\]

    From a computational standpoint, the two adjoint processes behave differently.  
The non-adapted adjoint $\Yn_t$ is obtained by solving a backward ordinary differential equation along each simulated trajectory, which makes it straightforward to compute once the forward trajectory $X_t$ is known.  
In contrast, the adapted adjoint $(Y_t,Z_t)$ is numerically challenging to simulate because the BSDE~\eqref{eq:bsde} imposes a terminal condition at time $T$, while the filtration $\{\mathcal F_t\}$ is generated forward in time. 
% is coupled with the SDE for the state dynamics.  
% A key difficulty arises because the filtration $\{\mathcal F_t\}$ is generated forward in time by the Wiener process, while the BSDE imposes a terminal condition at time $T$.  
% Consequently, solving BSDEs requires simultaneously satisfying forward dynamics and backward constraints, making the problem both numerically and conceptually challenging.
Our proposed methodology deals with this issue.

\section{Proposed methodology}\label{sec:method}
Our numerical approach is based on a time-reversed formulation of the BSDE~\eqref{eq:bsde}.  This construction relies on two key ingredients: (i) the connection between BSDEs and certain partial differential equation (PDE), and (ii) the time-reversal of diffusion processes. To present this construction, we assume that the control input is given in feedback form,
\begin{align}\label{eq:feedback}
    U_t = k(t,X_t).
\end{align}
where $k:[0,T]\times \Re^n \to \Re^n$ is a feedback policy. Under this assumption, the controlled diffusion~\eqref{eq:fsde-modified} becomes a Markov process whose dynamics are fully determined by the state $X_t$. This Markov structure allows the solution of the BSDE to be characterized through a PDE representation and enables the derivation of a time-reversed formulation that forms the basis of our numerical algorithm.

\subsection{PDE solution of the BSDE}
We begin by recalling the well-known connection between BSDE and PDE~\cite{peng1991probabilistic,ma1994solving}. Let  $\phi:[0,T]\times \R^n \to \R^n$ be a sufficiently smooth function satisfying the terminal value problem 
\begin{align}
    \frac{\partial \phi}{\partial t}(t,x) + \mathcal L \phi(t,x)  + h(t,x) 
     = 0, \quad  \phi(T,x) = \nabla \ell_f(x),\label{eq:pde}
\end{align}
where the second-order differential operator $\mathcal L$ is defined by
\begin{align*}
\mathcal L \phi (t,x)&:=  \frac{\partial \phi}{\partial x}(t,x)^\top \big(f(t,x)+g(t,x)k(t,x)\big) \\&\quad+ \frac{1}{2}\trace\big(G(t,x)\frac{\partial^2\phi}{\partial x^2}(t,x)\big),
\end{align*}
with $G(t,x):=g(t,x)g(t,x)^\top$, and
\begin{align}\label{eq:h}
h(t,x) &:=\frac{\partial H}{\partial  x}\big(t,x,\phi(t,x),g(t,x)^\top \partial_x\phi(t,x),k(t,x)\big). 
\end{align}
Applying It\^o's formula to $\phi(t,X_t)$ along the controlled diffusion~\eqref{eq:fsde-modified} shows that the processes
\begin{equation}
    (Y_t,Z_t) =  \big(\phi(t,X_t),\,g(t,X_t)^\top \frac{\partial \phi}{\partial x}(t,X_t)\big)
\end{equation}
solves the BSDE~\eqref{eq:bsde}. Throughout the paper we assume that the PDE~\eqref{eq:pde} admits a sufficiently regular solution.  
Such existence and regularity results hold under standard conditions on the coefficients of the diffusion and the Hamiltonian; see, e.g.,~\cite[Ch. 7 Sec. 4]{yong1999stochastic}. While solving the PDE directly is generally intractable in high dimensions, this representation provides the key insight needed to construct a time-reversed stochastic representation of the BSDE.

\subsection{Time-reversed diffusion}
Let $\Xr:=\{\Xr_t :=X_{T-t};\,0\leq t\leq T\}$ denote 
the time-reversal of the diffusion process $X$. According to the time-reversal theory for diffusion processes~\cite{haussmann1986time,cattiaux2023time}, the time-reversed process $\Xr$ satisfies the SDE
\begin{equation}\label{eq:tr-fsde}
    \ud \Xr_t = \fr(t, \Xr_t) \ud t  + \gr(t, \Xr_t)\big(\Ur_t \ud t+ \ud \Wr_t\big),
\end{equation}
where $\Wr_t$ is a standard $n$-dimensional Wiener process (independent of $W_t$). The coefficients of the reversed diffusion are given by
\begin{align*}
    \fr(t,x) &:= -f(T-t,x) + s(T-t,x),\\
    \gr(t,x) &:= - g(T-t,x),\qquad
    \Ur_t := k(T-t,\Xr_t),
\end{align*}
where $s: [0,T] \times \R^n \to \R^n$ denotes the {\it score function} associated with the marginal distribution of $X_t$. The score function  is defined component-wise as
\begin{equation}
    s_i(t,x) := \frac{1}{p(t,x)} \sum_{j=1}^n \frac{\partial}{\partial x_j}\big(G_{i,j}(t, x)p(t,x)\big),
\end{equation}
where $p(t,x)$ denote the probability density function of $X_t$, $G_{i,j}(t, x)$ is the $(i,j)$-entry of the matrix $G(t,x)$. 

In practice, the score function is not known explicitly and must be approximated from samples of the forward diffusion process.  
A common approach is to estimate the score by solving the \emph{implicit score matching} problem 
\begin{equation}\label{eq:score-opt}
    \min_\psi\; \mathbb E\left[\frac{1}{2}\|\psi(t,X_t)\|^2 + \sum_{i,j=1}^nG_{i,j}(t,X_t) \frac{\partial \psi_i}{\partial x_j}(t,X_t)\right],
\end{equation}
where the expectation is taken with respect to both $t\sim\mathrm{Unif}[0,T]$ and the state $X_t$ of the forward diffusion~\cite{song2019generative,hyvarinen2005estimation}.
The time-reversed diffusion will allow us to construct a time-reversed BSDE which is described next. 

\begin{algorithm}[htbp]
\caption{Time-Reversal BSDE solver} 
\begin{algorithmic}[1]
\STATE \textbf{Input:}  sample size $N$, step-size $\Delta t$, iteration $J_f$, function classes $\Psi$ and $\Phi$, initial distribution $q_0$.   
\STATE  Initialize $\{X_0^i\}_{i=1}^N \sim q_0$
\STATE Obtain $\{X^i_t\}_{i=1}^N$ by simulating \eqref{eq:fsde-modified}
\STATE Learn the score function $s(t,\cdot)$ through the implicit score matching optimization \eqref{eq:score-opt}
\STATE Set $\Xr_0^i = X_T^i$ for $i=1,\ldots,N$. 
\STATE  Obtain $\{\Xr^i_t\}_{i=1}^N$ by simulating  \eqref{eq:tr-fsde}
\STATE Initialize the function $\phi \in \Phi$
\FOR{$j \in \{1,2,\dots,J_f\}$}
\STATE $\tilde Y_0^i =  \nabla \ell_f (\Xr_0^i)$
\STATE Obtain $\{\Yr^i_t\}_{i=1}^N$ by simulating  \eqref{eq:tr-bsde}
 \STATE Update $\phi$ using ADAM optimizer minimizing~\eqref{eq:phi-opt}   
 %\STATE$\phi^{j+1}(T-t,\cdot) = \argmin_{\phi^{j+1} \in \Phi} \sum_{i=1}^N\|\tilde Y_t^i - \phi^{j+1}(T-t,\tilde X_t^i)\|^2$
\ENDFOR
 \STATE \textbf{Output:}  $\{\phi(t,\cdot)\}_{t\in \{0,\Delta t,\ldots,T\} }$
\end{algorithmic}
\label{alg:phi}
\end{algorithm}

\subsection{Time-reversed BSDE}
The time-reversed BSDE is obtained by applying It\^o's formula to the process
$\phi(T-t,\Xr_t)$, where $\phi$ is the solution of the PDE~\eqref{eq:pde} and 
$\Xr_t$ evolves according to the time-reversed diffusion~\eqref{eq:tr-fsde}. 
After simplification, this yields
\begin{align*}
    \ud \phi(T-t,\Xr_t) = &\left(h(T-t,\Xr_t) + c(T-t,\Xr_t) \right) \ud t \\&+ \frac{\partial \phi}{\partial x}(T\!-\!t,\Xr_t)^\top \gr(t,\Xr_t)\ud \Wr_t
\end{align*}
where $h$ is the partial derivative of the Hamiltonian as defined in~\eqref{eq:h} and \[c(t,x):=\frac{\partial \phi}{\partial x} (t,x)^\top s(t,x) +\trace\big(G(t,x)\frac{\partial^2\phi}{\partial x^2}(t,x)\big).\]
Define the processes
\begin{equation}
    (\Yr_t,\Zr_t):=\big(\phi(T-t,\Xr_t),\gr(t,\Xr_t)^\top \frac{\partial}{\partial x}\phi(T-t,\Xr_t)\big). 
\end{equation}
Then $\Yr_t$ satisfies the SDE
\begin{align}
    \ud \Yr_t &= \frac{\partial H}{\partial x}\big(T-t,\Xr_t,\Yr_t,\Zr_t,\Ur_t\big)\ud t + c(T-t,\tilde X_t)\ud t + \tilde Z_t^\top \ud \tilde W_t, \nonumber \\
    \Yr_0 &= \nabla \ell_f (\Xr_0),\label{eq:tr-bsde}
\end{align}
which can be interpreted as the time-reversal of the BSDE~\eqref{eq:bsde}.  
By construction, the path $\{(\Yr_t,\Zr_t);0\leq t \leq T\}$ has the same probability law as $\{( Y_{T-t}, Z_{T-t});0\leq t \leq T\}$.

\subsection{Numerical procedure}
The key advantage of the representation~\eqref{eq:tr-bsde} is that, unlike the original BSDE~\eqref{eq:bsde}, it does not impose a terminal constraint.  
Instead, the process can be simulated forward in time starting from the initial condition $\Yr_0=\nabla \ell_f(\Xr_0)$.  
This property allows the BSDE to be approximated using standard forward simulation techniques. 

The simulation of~\eqref{eq:tr-bsde}, however, requires the knowledge of $\phi$. To overcome this challenge, we follow an iterative procedure, where we start with an initial guess for $\phi$, simulate~\eqref{eq:tr-bsde}, and update $\phi$ through the optimization 
\begin{equation}\label{eq:phi-opt}
    \min_{\phi \in \Phi}\; \sum_{i=1}^N \|\Yr_t^i - \phi(T-t,\Xr_t^i)\|^2
\end{equation}
where $\{(\Xr_t^i,\Yr_t^i)\}_{i=1}^N$ are simulated trajectories generated by the coupled system~\eqref{eq:tr-fsde}--\eqref{eq:tr-bsde}. The time-reversed BSDE~\eqref{eq:tr-bsde} therefore provides the basis for a practical numerical algorithm for solving the original BSDE.  
% \subsection{Numerical procedure}
% The algorithm starts with simulating $N$ independent realizations of the state process $\{X_t^i\}_{i=1}^N$ according to~\eqref{eq:fsde} by Euler-Maruyama method over a discretization of the time-domain. Then, approximating the score function by solving the optimization problem~\eqref{eq:score-opt}. Next, the time-reversed state and adjoint processes \eqref{eq:tr-fsde}~\eqref{eq:tr-bsde} are simulated backward, while simultaneously the minimization problem~\eqref{eq:phi-opt} is solved to obtain the function $\phi$.
The details appear in Algorithm~\ref{alg:phi}. 

\begin{remark}
The proposed methodology remains applicable even when the score function $s$ 
is only approximately known. In particular, the function 
$\phi$ can still be learned by solving the regression problem~\eqref{eq:phi-opt}. In practice, the accuracy of the learned function $\phi$ is highest in regions of the state space where the reversed trajectories are sampled more frequently. The role of the score function is to ensure that these trajectories remain consistent with the forward diffusion, thereby concentrating samples in regions that are most relevant for the SOC problem and improving the accuracy of $\phi$ where it matters most for the optimization.
\end{remark}

\subsection{Iterative procedure and its convergence analysis}\label{sec:convergence}
The more detailed iterative procedure and its convergence analysis are presented in this section.
% We present an iterative procedure to simulate~\eqref{eq:tr-bsde}. 
Starting from an initialization $\phi^0$, at 
iteration $j$ the algorithm simulates~\eqref{eq:tr-bsde} with $\phi$ 
replaced by $\phi^{j-1}$:
\begin{align}\nonumber
    \ud \Yr^j_t = &\frac{\partial H}{\partial x}\big(T\!-\!t,\Xr_t,\Yr^j_t,\Zr^{j-1}_t,\Ur_t\big)\ud t \!+\! c^{j-1}(T\!-\!t,\Xr_t)\ud t \\&+ {\Zr^{j-1}}_t \ud \tilde W_t,\label{eq:tr-bsde-j}
\end{align}
where $\Zr^{j-1}_t = \frac{\partial \phi^{j-1}}{\partial  x}(T-t,\Xr_t)g(T-t,\Xr_t)$ and  \[c^{j-1}(t,x):=\frac{\partial \phi^{j-1}}{\partial x}(t,x) s (t,x) +\trace\big(G(t,x)\frac{\partial^2\phi^{j-1}}{\partial x^2}(t,x)\big).\]
The simulated trajectories are then used to update $\phi^j$ via the regression
\begin{equation}\label{eq:phi-j}
    \phi^j := \argmin_{\psi \in \mathcal{G}}\, \mathbb{E}\!\left[\int_0^T
    \|\tilde{Y}^j_{T-t} - \psi(t,\tilde{X}_{T-t})\|^2\,\ud t\right],
\end{equation}
where $\mathcal{G}$ is a user-specified function class. In practice, the 
expectation is approximated by simulating $N$ independent trajectories of 
$(\tilde{X}_t, \tilde{Y}^j_t)$ via Euler--Maruyama discretization 
of~\eqref{eq:tr-fsde}--\eqref{eq:tr-bsde-j}.

 For the convergence analysis, we set aside discretization and optimization 
errors and assume both the SDE~\eqref{eq:tr-bsde-j} and the 
minimization~\eqref{eq:phi-j} are solved exactly. Define the 
operators
\begin{align*}
    \mathcal A_t \phi &: = \frac{\partial \phi}{\partial x} (f+gk-s) - \frac{1}{2}\trace(G\frac{\partial^2 \phi}{\partial x^2}) + \frac{\partial f}{\partial x}^\top \phi,\\\mathcal A^\prime_t \phi  &:= \frac{\partial \phi}{\partial x} s + \trace(G\frac{\partial^2 \phi}{\partial x^2}),
\end{align*}
and the projection onto $\mathcal{G}$:
\begin{align*}
    \Pi^{\mathcal G} (\phi) := \argmin_{\psi \in \mathcal G} \Expect[\int_0^T\|\phi(t,\tilde X_{T-t})-\psi(t,\tilde X_{T-t})\|^2\ud t]. 
\end{align*}
\begin{proposition}
Assume $g(t,x)=g(t)$. Then the $j$-th iterate satisfies
\begin{equation}\label{eq:Yr}
   \Yr_t = \psi^j(T-t,\Xr_t) + M_t,\quad \phi^j = \Pi^{\mathcal G} (\psi^j)
\end{equation}
where $M_t$ is a martingale with respect to the filtration generated by 
$\tilde{W}_t$, and $\psi^j$ solves the PDE
\begin{equation}\label{eq:pde-psi}
    \frac{\partial \psi^j}{\partial t} + \mathcal A_t\psi^j + \mathcal A^\prime_t \phi^{j-1} = 0, ~~\psi^j(T,x) = \nabla \ell_f(x). 
\end{equation}
% The iterate $\phi^j$ is given by the projection
% \begin{align}\label{eq:phi-j-proj}
%     \phi^j = \Pi_{\mathcal F} (\psi^j)
% \end{align}
\end{proposition}
\begin{proof}
Applying It\^{o}'s formula to $\psi^j(T-t,\tilde{X}_t)$ and using 
the PDE~\eqref{eq:pde-psi} gives
\begin{align*}
    \ud \psi^j(T-t,\tilde X_t) 
    = &\mathcal A^\prime_t \phi^{j-1}(T-t,\tilde X_t) \ud t + \frac{\partial f}{\partial x}^\top \psi^j(T-t,\tilde X_t) \\&+\frac{\partial \psi^j}{\partial x}g(T-t,\tilde X_t) \ud \tilde W_t.
\end{align*}
Subtracting from~\eqref{eq:tr-bsde-j} and using $c^{j-1}(t,x)=\mathcal A^\prime_t \phi^{j-1}(t,x)$ and $\frac{\partial H}{\partial x}\big(T-t,\Xr_t,\Yr^j_t,\Zr^{j-1}_t,\Ur_t\big)=\frac{\partial f}{\partial x}^\top(T-t,\tilde X_t)\tilde Y_t$ shows that the difference $\tilde{Y}^j_t - \psi^j(T-t,\tilde{X}_t)$ satisfies
\begin{align*}
    \ud \left[\Yr^j_t - \psi^j(T-t,\tilde X_t)\right] = \frac{\partial f}{\partial x}^\top\left[\Yr_t - \psi^j(T-t,\tilde X_t)\right] + \ud \tilde M_t
\end{align*}
where $\ud \tilde M_t = \Zr^{j-1} \ud \tilde W_t-\frac{\partial \psi^j}{\partial x}g(T-t,\tilde X_t) \ud \tilde W_t$. Since $\tilde Y^j_0 - \psi^j(T,\tilde X_0) = \nabla \ell_f(\tilde X_0)-\nabla \ell_f(\tilde X_0)=0$, the difference $\Yr^j_t - \psi^j(T-t,\tilde X_t)$ is a Martingale concluding~\eqref{eq:Yr}. Substituting 
into~\eqref{eq:phi-j} gives the expression for $\phi^j$.
\end{proof}

Our next goal is concerned with the convergence of the iterates $\phi^j$. Let $\psi^\star$ denote the fixed point satisfying
\begin{equation*}
    \frac{\partial \psi^\star}{\partial t} + \mathcal A_t \psi^\star + \mathcal A'_t \phi^\star  =0,\quad \psi^\star(T,x) = \nabla \ell_f(x)
\end{equation*}
where $\phi^\star = \Pi^{\mathcal G}(\psi^\star)$. 
When $\Pi_{\mathcal{G}}$ is the identity, $\psi^\star = \phi^\star$ recovers 
the solution to PDE~\eqref{eq:pde}. 

To state the convergence result, we specialize to the case where $\mathcal{G}$ 
is a finite-dimensional function class. Specifically, let $\mathcal{H}:=L^2([0,T]\times\Re^n; \Re^n)$ be equipped with the inner product
\[
    \langle \phi, \psi \rangle_{\mathcal{H}} 
    := \int_0^T\int_{\mathbb{R}^n} \phi(x)^\top \psi(x)\, \rho_t(x)\, \ud x\,\ud t ,
\]
where $\rho_t$ is the marginal density of $\tilde{X}_t$. Assume $\mathcal{G} = 
\mathrm{span}\{\varphi_1,\ldots,\varphi_K\}$, where $\{\varphi_k\}_{k=1}^K 
\subset \mathcal{H}$ are orthonormal basis functions. Under this assumption, 
the projection $\Pi^{\mathcal{G}} : \mathcal{H} \to \mathcal{G}$ is the 
orthogonal projection
$
    \Pi_{\mathcal{G}}(\phi) = \sum_{k=1}^K \langle \phi, \varphi_k 
    \rangle_{\mathcal{H}}\, \varphi_k,
$
which is a bounded linear operator on $\mathcal{H}$. For any $t\in [0,T]$ define the norm $\|\phi\|_{\mathcal H_t}^2:=\int \|\phi(x)\|^2\rho(t,x)\ud x$. Define the constant
\begin{equation}\label{eq:M}
    M := \sup_{t,s \in [0,T]} \left\| \Pi_t^{\mathcal{G}}\, 
    \mathcal{P}_{t,s}\, \mathcal{A}'_s\,  \Pi_s^{\mathcal{G}} 
    \right\|_{\mathrm{op}},
\end{equation}
where $\|\cdot\|_{\mathrm{op}}$ denotes the operator norm induced by $\| \cdot\|_{\mathcal H_s} \to \|\cdot\|_{\mathcal H_t}$, $\Pi^{\mathcal G}_t(\phi):=\Pi^{\mathcal G}(\phi)(t,\cdot)$ is the evaluation of the projection at time $t$, and 
$\mathcal{P}_{t,s}$ is the semigroup associated with $\mathcal{A}_t$. Since 
$\mathcal{G}$ is finite-dimensional, the composition is represented, 
in the basis $\{\varphi_k\}$, by the $K\times K$ matrix. 
Under smoothness assumptions on the basis functions, each matrix entry is 
finite and $M < \infty$.

\begin{proposition}\label{prop:convergence}
Under the above assumptions, the iterates $\phi^j$ converge to the fixed 
point $\phi^\star$ at a superlinear rate:
\begin{align*}
    \sup_{t\in[0,T]}\|\phi^j(t,\cdot) - \phi^\star(t,\cdot)\|_{\mathcal{H}_t} 
    \leq \frac{(MT)^j}{j!}\sup_{t\in[0,T]}
    \|\phi^0(t,\cdot) - \phi^\star(t,\cdot)\|_{\mathcal{H}_s}.
\end{align*}
In particular, $\phi^j \to \phi^\star$ as $j\to\infty$ for any initialization 
$\phi^0$.
\end{proposition}

% The following result establishes 
% convergence of the iterates. To state the result, we consider a 
% \begin{proposition}
%     Assume $\Pi^{\mathcal F}$ is a linear operator and $\exists M>0$ such that 
% \begin{align}\label{eq:norm-bound}
%     \| \Pi^{\mathcal F}_t \mathcal P_{t,s} \mathcal A'_s  \Pi^{\mathcal F}_s \| \leq  M \quad \forall t,s \in [0,T],
% \end{align}
% where $\mathcal P_{t,s}$ is the semi-group associated with $\mathcal A_t$, and $\Pi^{\mathcal F}_t(\phi):=\Pi^{\mathcal F}(\phi)(t,\cdot)$ is the evaluation of the projection at time $t$. Then,   
% \begin{align*}
%    \sup_{t\in [0,T]} \| \phi^J(t,\cdot) - \phi^\star(t,\cdot)\| \leq \frac{(MT)^J}{J!} \sup_{s\in [0,T]} \|\phi^0(t,\cdot) - \phi^\star(t,\cdot)\|.
% \end{align*}
% \end{proposition}
\begin{proof}The proof follows a Picard iteration argument. By linearity of $\Pi_{\mathcal{G}}$, the error $e^j = \psi^j - \psi^\star$ 
solves
\begin{equation}\label{eq:error-pde}
    \frac{\partial e^j}{\partial t} + \mathcal A_t e^j + \mathcal A_t^\prime \Pi^{\mathcal G}(e^{j-1}) = 0,~~ e^j(T,x) = 0,
\end{equation}
which yields
\begin{align*}
    e^j(t,\cdot) = \int_0^t \mathcal P_{t,s} \mathcal A'_s \Pi^{\mathcal G} (e^{j-1})(s,\cdot) \ud s. 
\end{align*}
Setting $\tilde e^j:= \Pi^{\mathcal G}(e^j)=\phi^j-\phi^\star$ and $\mathcal L_{t,s}:= \Pi^{\mathcal G}_t \mathcal P_{t,s}\mathcal A^\prime_s  \Pi^{\mathcal G}_s$,
\begin{align*}
    \tilde e^j(t,\cdot) = \int_0^t \mathcal L_{t,s} \tilde e^{j-1}(s,\cdot) \ud s. 
\end{align*}
Iterating it from $j=J$ to $j=1$ concludes 
\begin{align*}
    \tilde e^J(s_J,\cdot) = \int_0^{s_J} \ldots \int_0^{s_1}\mathcal L_{s_J,s_{J-1}}  \ldots \mathcal L_{s_2,s_{1}} \tilde e^0(s_1,\cdot)\ud s_1\ldots\ud s_J. 
\end{align*}
Taking the norm of both sides and applying~\eqref{eq:M} yields
\begin{align*}
    \| \tilde e^J(t,\cdot)\| \leq \frac{(MT)^J}{J!} \sup_s \|\tilde e^{0}(s,\cdot)\|,\quad \forall t \in[0,T],
\end{align*}
concluding the proof. 
\end{proof}
We note that there remains a gap between this theoretical result and the 
practical algorithm, where $\mathcal{G}$ is parameterized by a neural network 
and the minimization~\eqref{eq:phi-j} is solved approximately; closing this 
gap is left for future work.

%%%%%%%%%%%%%%%%%%%%%%%%%%%%%%%%%%%%%%%%%%%%%%%%%%%%%%%%%%%%%%%%%%%%%%%%%%%%%%%%

\section{Numerical Experiments}\label{sec:num}
In this section, we present numerical results for four experiments: a linear model, a stochastic pendulum, the full fine-tuning procedure on a toy diffusion model, and a double-well benchmark. The first two experiments are designed to assess the ability of the proposed framework to compute the adapted adjoint process and to compare its accuracy, for gradient estimation, against the non-adapted adjoint process. 
The third and fourth experiments compare the performance of the proposed TR-BSDE method against adjoint matching, which relies on the non-adapted adjoint process.
% The third and fourth experiments demonstrate the use of our BSDE solver for fine-tuning a diffusion model, and compares its performance with adjoint matching, which relies on the non-adapted adjoint process. 
The code for reproducing the results and the details of the numerical procedures are available online~\footnote{\url{https://github.com/YuhangMeiUW/BSDE-Gradient}}.

% Both original system dynamics \eqref{eq:fsde-modified} and time-reversal system dynamics are simulated by Euler-Maryama method. Notice that adapted, non-adapted adjoint process, and time-reversed adapted adjoint process are all linear in adjoint state. In order to improve the simulation accuracy, we use the exponential integrator for these three adjoint processes.

\begin{algorithm}[htbp]
\caption{Fine-tuning the diffusion model} 
\begin{algorithmic}[1]
\STATE \textbf{Input:}  Iteration number $k_f$, optimization iteration number $o_f$, initial distribution $q_0$ with parameters $\mu$ and $Q$, initial control law $k(t,x) =0$.
\FOR{$\mathrm k \in \{1,2,\dots,k_f\}$}
\STATE Obtain the function  $\phi$ from Algorithm \ref{alg:phi} 
\STATE Update the feedback control law accoding to $k(t,x) = - g(t)^\top\phi(t,x)$
\IF{$\mathrm k \bmod 5 = 0$}
% \STATE Compute gradient $\frac{\partial}{\partial \mu}J$ and $\frac{\partial}{\partial \mu}J$~\eqref{eq:ini-grad} with learned function $\phi(t,\cdot)$, and update initial distribution parameter $\theta$ through ADAM optimizer for $o_f$ iteration
\STATE Update $\mu$ and $Q$ according to~\eqref{eq:ini-grad} using the ADAM optimizer for $o_f$ iterations
\ENDIF
\ENDFOR
\STATE \textbf{Output:} Feedback control law $k(t,x)$, optimal initial distribution $q_0$
\end{algorithmic}
\label{alg:fine-tune}
\end{algorithm}

\begin{algorithm}[htbp]
\caption{Projected non-adapted adjoint process} 
\begin{algorithmic}[1]
\STATE \textbf{Input:}  sample size $N$, step-size $\Delta t$, iteration $J_f$, function class $\Phi$, initial distribution $q_0$.   
\STATE  Initialize $\{X_0^i\}_{i=1}^N \sim q_0$
\STATE Obtain $\{X_t^i\}_{i=1}^N$ by simulating~\eqref{eq:fsde-modified}
\STATE Initialize function $\phi \in \Phi$
% \FOR{$j \in \{1,2,\dots,J_f\}$}
\STATE $\Yn_T^i = \nabla \ell_f (X_T^i)$
\STATE Obtain $\{\Yn_t^i\}_{i=1}^N$ by simulating~\eqref{eq:non-ad}
 \STATE Update $\phi$  using ADAM optimizer minimizing~\eqref{eq:pnaa-opt}
 % \STATE$\phi^{j+1}(t,\cdot) = \argmin_{\phi^{j+1} \in \Phi} \sum_{i=1}^N\|\Yn_t^i - \phi^{j+1}(t, X_t^i)\|^2$
% \ENDFOR
 \STATE \textbf{Output:}  $\{\phi(t,\cdot)\}_{t\in \{0,\Delta t,\ldots,T\} }$
\end{algorithmic}
\label{alg:pnaa}
\end{algorithm}

\begin{algorithm}[htbp]
\caption{Adjoint matching fine-tuning} 
\begin{algorithmic}[1]
\STATE \textbf{Input:} Iteration number $k_f$, control input iteration number $J_f$, optimization iteration number $o_f$, function class $\Phi$, initial distribution $q_0$ with parameters $\mu$ and $Q$, initial control law $k(t,x) =0$.
\FOR{$k \in \{1,2,\dots,k_f\}$}
\FOR{$j \in \{1,2,\dots,J_f\}$}
\STATE Obtain $\{X_t^i\}_{i=1}^N$ by simulating \eqref{eq:fsde-modified} 
\STATE $\Yn_T^i = \nabla \ell_f(X_T^i)$
\STATE Obtain $\{\Yn_t^i\}_{i=1}^N$ by simulating~\eqref{eq:non-ad}
\STATE Update control law $k(t,x)$ using ADAM optimizer minimizing
 $\frac{1}{N}\sum_{i=1}^N \|k(t,X_t^i) + g(t)^\top \Yn_t^i\|$
\ENDFOR
\IF{$k \bmod 5 = 0$}
\STATE Update $\mu$ and $Q$ according to~\eqref{eq:ini-grad} with $\{\Yn_0^i\}_{i=1}^N$ using the ADAM optimizer for $o_f$ iteration
\ENDIF
\ENDFOR
\STATE \textbf{Output:} Feedback control law $k(t,x)$, optimal initial distribution $q_0$
\end{algorithmic}
\label{alg:ad-mat}
\end{algorithm}

\subsection{Verification in linear quadratic setting}\label{sec:num-lin}
We consider the $2$-dimensional linear system 
\begin{equation}
    \ud X_t = AX_t\ud t + \epsilon B\ud W_t, ~~X_0 =\xi \sim q_0,
\end{equation}
and the objective function with the quadratic terminal cost 
\begin{equation}
    \Expect[J(\xi)]:=\Expect \left[\frac{1}{2}X_T^\top Q_fX_T\right],
\end{equation}
where $\epsilon>0$  denotes the strength of the noise, 
\[A = \begin{bmatrix}0&1\\-1&-0.5\end{bmatrix}, \quad B = \begin{bmatrix}1&0\\0&1\end{bmatrix}, \quad Q_f = \begin{bmatrix}1&0\\0&1\end{bmatrix},\] 
$q_0$ is standard Gaussian, and $T=2$. We simulate and compare two numerical methods to estimate the derivative $\frac{\partial \Expect[J(\xi)]}{\partial \xi}$. The ground-truth is  $\phi(0,\xi)$ where $\phi$ solves the PDE~\eqref{eq:pde}. In the linear case, the exact solution is $\phi(t,x)=G(t)x$ where $G(t)$ solves the differential Riccati equation~\cite[Eq. 15]{mei2025time}. 

\medskip 
\noindent
{\bf Time-reversed BSDE (TR-BSDE):} To estimate $\phi(0,x)$, we use the iterative procedure that alternates between simulating the time-reversed BSDE~\eqref{eq:tr-bsde} and solving the optimization ~\eqref{eq:phi-opt}. Although the ground-truth function $\phi$ is linear in this example, we parameterize it with a neural network in order to validate the proposed training procedure. We perform $10$ outer iterations. In each iteration, we simulate~\eqref{eq:tr-bsde}, then train the neural network for $\phi$ over $2000$ optimization steps, and then use the updated $\phi$ in the next round of BSDE simulation, as described in Algorithm~\ref{alg:phi}.

\medskip 
\noindent
{\bf Projected non-adapted adjoint process (PNAA):}   To estimate $\phi(0,x)$, we use the relation $\phi(0,\xi)=Y_0=\Expect[\Yn_0 \mid \mathcal F_0]$, where $\Yn$ denotes the non-adapted adjoint process. Specifically, we simulate $N$ sample trajectories $\{(X^i,\Yn^i)\}_{i=1}^N$ with random initial conditions $X^i_0=\xi^i\sim q_0$, and then estimate $\phi$ by solving the regression problem 
\begin{equation}\label{eq:pnaa-opt}
    \min_{\phi \in \Phi}\sum_{i=1}^N \| \Yn^i_t - \phi(t,X^i_t)\|^2.
\end{equation}
The procedure is detailed in Algorithm~\ref{alg:pnaa}.
To ensure a fair comparison with TR-BSDE, we use the same neural network architecture and the same total number of training iterations.
% , and $q_0^\xi$ is an empirical distribution 
% \begin{equation*}
%     q_0^\xi = \frac{1}{N}\sum_i^N \delta_{\xi^i}, ~~ \{\xi^i\}_{i=1}^N\sim \mathcal N(0,I)
% \end{equation*}

% We want to optimize the following cost
% \begin{equation*}
%     \bar J(\xi^1,\xi^2,\dots,\xi^N) = \sum_{i=1}^N J(\xi^i),
% \end{equation*}
% and the gradient of the cost with respect to the parameter $\xi^i$ can be expressed as
% \begin{equation*}
%     \frac{\partial}{\partial \xi^i} \bar J =  \frac{\partial}{\partial x}J(\xi^i)
% \end{equation*}
% The ground truth for $\frac{\partial}{\partial x}J(\xi^i)$ is $G_0\xi^i$, where $G_0$ is the value of $G_t$ at time $t=0$ and $G_t$ solves the following differential equation.
% \begin{equation}
%     \dot G_t + G_t^\top A + A^\top G_t = 0, ~~ G_T = Q_f
% \end{equation}
% Let time step $\Delta t=0.05$ and iteration number $J_f=10$, we use two neural networks to parameterize the score function and $\phi$ function and apply our algorithm.
% We use ADAM optimizer with step size $10^{-3}$ and optimize the initial points for 5000 iterations with phi and compare with the automatic differentiation method. The result is shown is Fig \ref{fig:lin-heat}.
% \begin{figure}[thpb]
%     \centering
%     \includegraphics[width=0.6\linewidth]{figs/linear_example_loss_landscape.pdf}
%     \caption{Optimization result with noise strength $\epsilon=1$ and sample number $N=2000$. Comparison between the Time-reversal BSDE and Auto Differentiation. The heat map represents the expected terminal cost corresponding to different initial states.}
%     \label{fig:lin-heat}
% \end{figure}
To quantify the accuracy of gradient estimation, we use the following mean-squared error (MSE) criterion:
\begin{equation}
    \text{MSE} = \frac{1}{N} \sum_{i=1}^N \left\| G(0)\xi^i - \phi(0,\xi^i)\right\|^2_2,\quad \xi^i \sim q_0,
\end{equation}
where $G_0\xi^i$ is the ground truth value and $\phi(0,\xi)$ is its estimate obtained from the two methods above. 
Figure~\ref{fig:lin-noise} illustrates the effect of the noise strength $\epsilon$ on the performance of the two estimators. As expected, the MSE increases with $\epsilon$ for both methods. However, the proposed TR-BSDE method consistently achieves lower error than PNAA across all noise levels, reflecting the fact that the non-adapted adjoint process yields inherently noisier gradient estimates. 
% adjoint method. For fair comparison, we simulate the non-adapted adjoint process and use a neural network with the same structure as our method to learn a $\phi$ function by regression, same as \eqref{eq:phi-opt}, between the non-adapted adjoint process and the forward process, and this is named the projected non-adapted adjoint process (PNAA) method.
% The results for varying noise strength are shown in Fig.~\ref{fig:lin-noise}, where the sample size is fixed at $N=2000$.
\begin{figure*}
    \centering
    \begin{subfigure}{0.32\textwidth}\centering
        \includegraphics[width=\hsize,trim={0 10 0 0},clip]{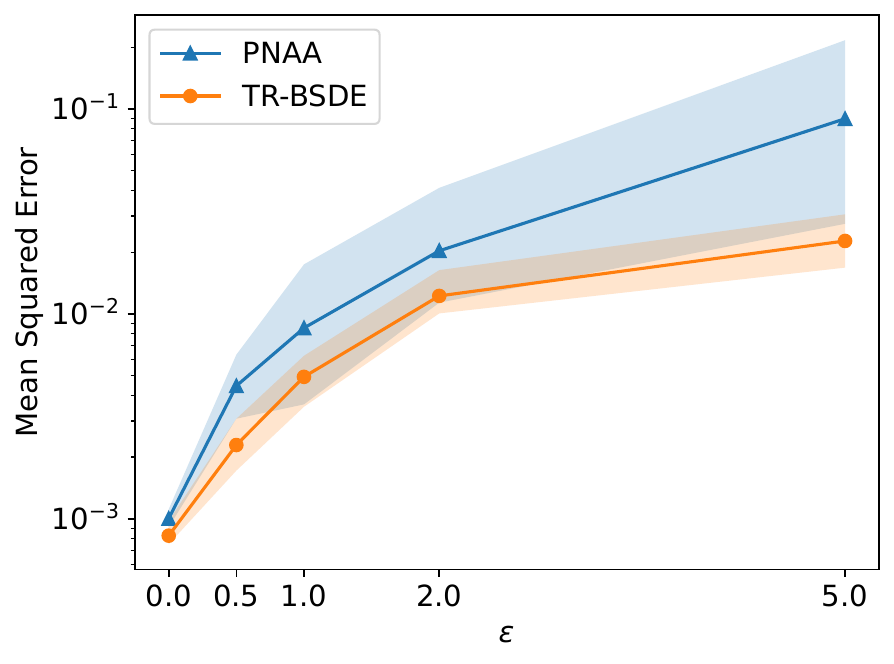}
        \caption{}
        \label{fig:lin-noise}
    \end{subfigure}
    \begin{subfigure}{0.315\textwidth}\centering
        \includegraphics[width=\hsize,trim={0 0 0 0},clip]{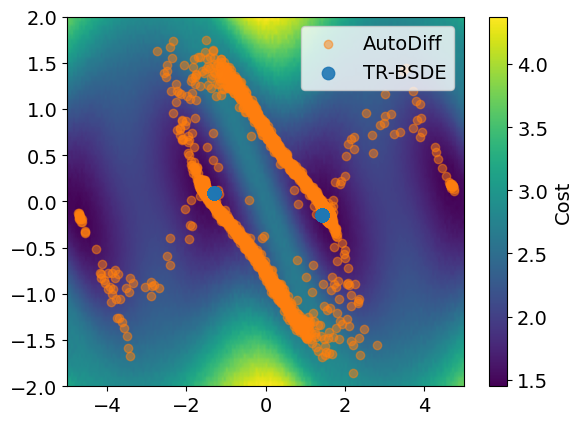}
        \caption{}
        \label{fig:IP-heat-ad}
    \end{subfigure}
    \begin{subfigure}{0.315\textwidth}\centering
        \includegraphics[width=\hsize,trim={0 0 0 0},clip]{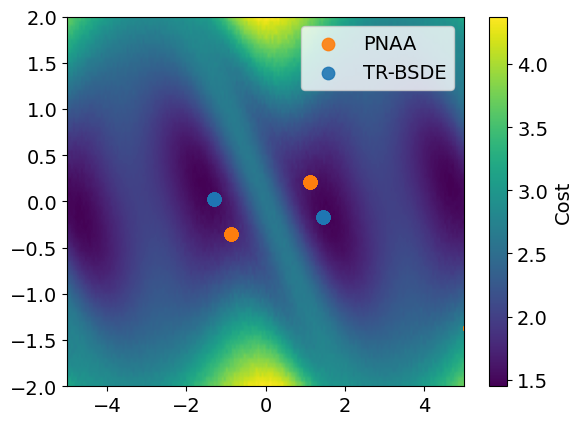}
        \caption{}
        \label{fig:IP-heat-pnaa}
    \end{subfigure}
    \caption{Numerical results for comparing the accuracy of gradient estimation methods. (a) influence of noise strength $\epsilon$ on MSE for the linear system in Sec.~\ref{sec:num-lin}. The solid line represents the mean, and the shaded region represents the range from the minimum to the maximum across $10$ experiments. (b-c) Optimal initial distribution for the nonlinear system in Sec.~\ref{sec:num-ip}. The heat map represents the optimization objective function and the points represent the support of the empirical distribution. 
    %Panel (b) compares  result of auto differentiation and TR-BSDE (c) Optimization results of PNAA and TR-BSDE. Since the PNAA result includes several points that drift far from the origin, we present a zoomed-in view of the region around the origin for clearer visualization. The heat map represents the expected terminal cost for different initial states.
    }
\end{figure*}
\begin{figure*}[thpb]
	\centering
    \begin{subfigure}{0.24\textwidth}\centering
        \includegraphics[width=\hsize,trim={0pt 0pt 0pt 0pt},clip]{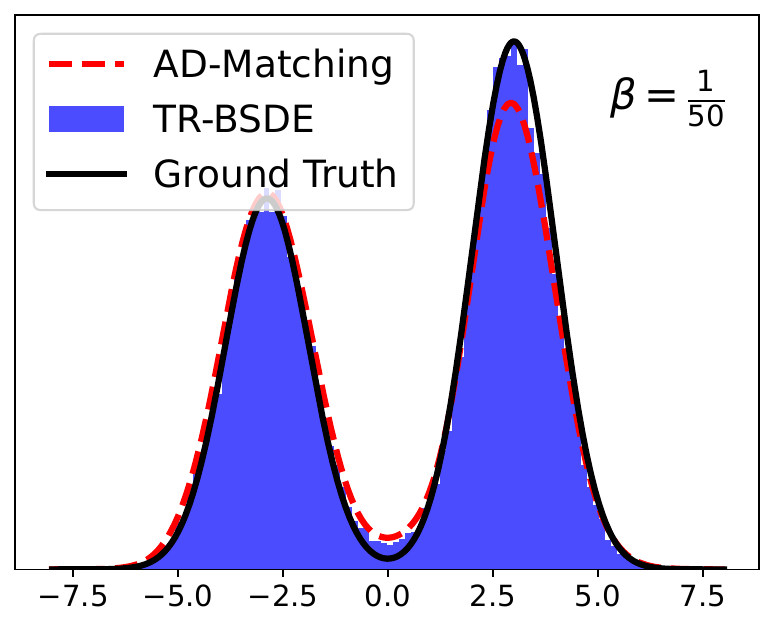}
        % \caption{}
    \end{subfigure}
    \begin{subfigure}{0.24\textwidth}
        \includegraphics[width=\hsize,trim={0pt 0pt 0pt 0 pt},clip]{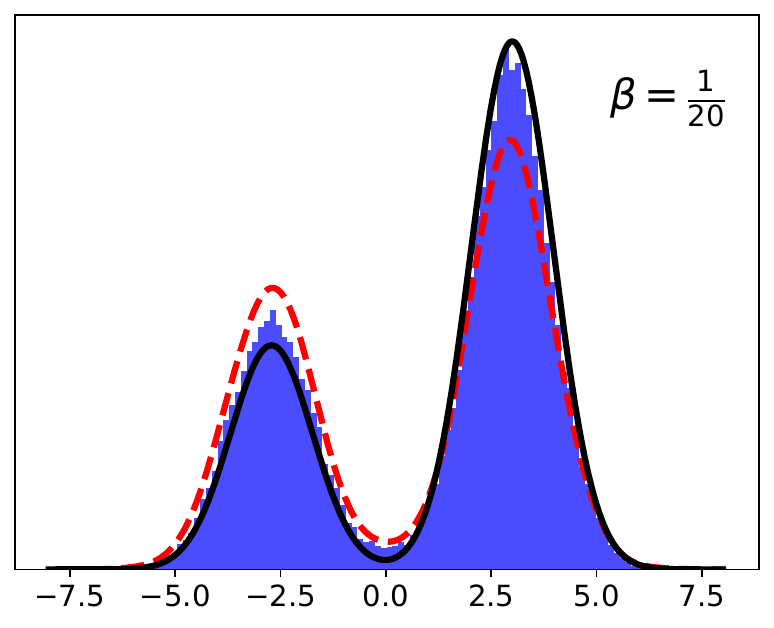}
        % \caption{}
    \end{subfigure}
    \begin{subfigure}{0.24\textwidth}
        \includegraphics[width=\hsize,trim={0pt 0pt 0pt 0 pt},clip]{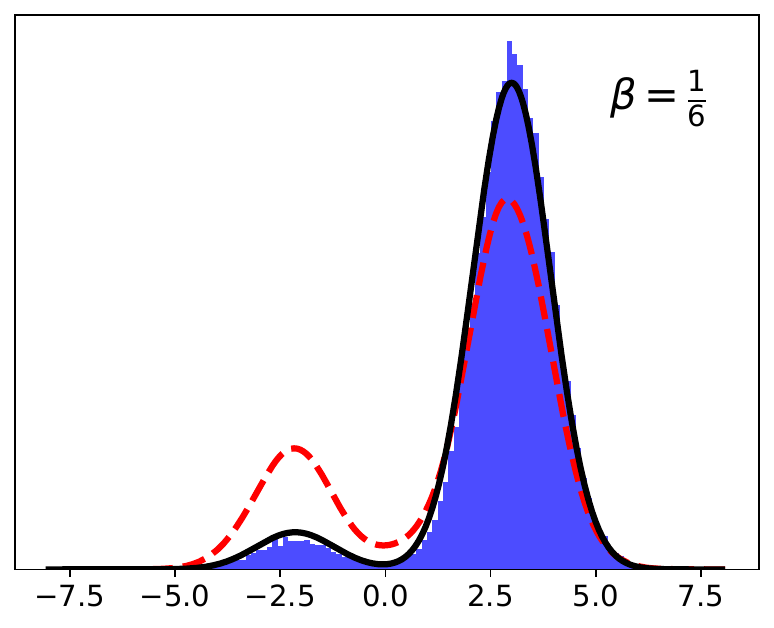}
        % \caption{}
    \end{subfigure}
    \begin{subfigure}{0.24\textwidth}
        \includegraphics[width=\hsize,trim={0pt 0pt 0pt 0 pt},clip]{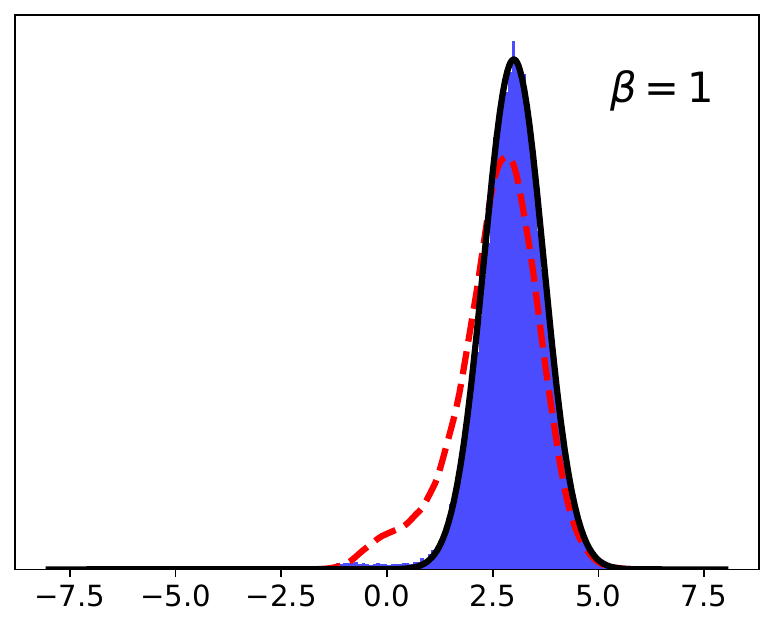}
        % \caption{}
    \end{subfigure}
    % \begin{subfigure}{0.24\textwidth}\centering
    %     \includegraphics[width=\hsize,trim={0pt 0pt 0pt 0pt},clip]{figs/finetune_admatching_temp50.0.pdf}
    %     \caption{}
    % \end{subfigure}
    % \begin{subfigure}{0.24\textwidth}
    %     \includegraphics[width=\hsize,trim={0pt 0pt 0pt 0 pt},clip]{figs/finetune_admatching_temp20.0.pdf}
    %     \caption{}
    % \end{subfigure}
    % \begin{subfigure}{0.24\textwidth}
    %     \includegraphics[width=\hsize,trim={0pt 0pt 0pt 0 pt},clip]{figs/finetune_admatching_temp6.0.pdf}
    %     \caption{}
    % \end{subfigure}
    % \begin{subfigure}{0.24\textwidth}
    %     \includegraphics[width=\hsize,trim={0pt 0pt 0pt 0 pt},clip]{figs/finetune_admatching_temp1.0.pdf}
    %     \caption{}
    % \end{subfigure}
	\caption{Numerical result for comparing TR-BSDE and adjoint matching on fine-tuning the diffusion model (Sec.~\ref{sec:num-diff}): The black solid line denotes the target tilted distribution, the red dashed line denotes the fine-tuned result from the adjoint matching method, and the blue histogram denotes the fine-tuned result from TR-BSDE method.}
    \label{fig:sample-compare}
    %Time horizon $T=1$, sample size $N=1000$, and noise scalar $\epsilon=0.3$}
\end{figure*}
% \begin{figure}[tbh]
%     \centering

%     % \begin{subfigure}{0.3\textwidth}
%         % \centering
%         \includegraphics[width=0.7\linewidth]{figs/stable_linear_example_trainphi_mse_nl_expnum10both.pdf}
%         \caption{Numerical results for comparing the two methods: Time-reversal BSDE and non-adapted adjoint process on the 2-dimensional linear system gradient estimation presented in Sec \ref{sec:num-lin}. The figure demonstrates the influence of noise strength $\epsilon$ on MSE with $N=2000$. The solid line represents the mean value, and the shaded region represents the range from the minimum to the maximum across 10 experiments.}
%         \label{fig:lin-noise}
%     % \end{subfigure}

%     % \vspace{0.5cm}

%     % \begin{subfigure}{0.3\textwidth}
%     %     \centering
%     %     \includegraphics[width=\linewidth]{figs/stable_linear_example_mse_trainphi_samplenum_expnum10.pdf}
%     %     \caption{}
%     %     \label{fig:lin-sample}
%     % \end{subfigure}

%     % \caption{Numerical results for comparing the two methods: Time-reversal BSDE and non-adapted adjoint process on the 2-dimensional linear system gradient estimation presented in Sec \ref{sec:num-lin}. (a) Influence of noise strength $\epsilon$ on MSE with $N=2000$. (b) Influence of sample size $N$ on MSE with $\epsilon=1$. The solid line represents the mean value and the shaded region represents the range from the minimum to the maximum across 20 experiments.}
% \end{figure}

\subsection{Optimal initial distribution for inverted pendulum}\label{sec:num-ip}
We consider the nonlinear stochastic pendulum
\begin{equation*}
    f(x) = \begin{bmatrix}x(2)\\\sin(x(1))-0.01x(2)\end{bmatrix}, ~~g(x) = \begin{bmatrix}
        0\\0.5
    \end{bmatrix},
\end{equation*}
and the objective 
\begin{equation}
    \Expect[J(\xi)] := \Expect \left[\frac{1}{2}X_T(2)^2 + 1-\cos(X_T(1)) \right],\quad \xi \sim q_0.  
\end{equation}
We study the problem of optimizing $\Expect[J(\xi)]$ with respect to the initial distribution $q_0$, assumed to have the empirical form
\begin{equation*}
    q_0 = \frac{1}{N}\sum_i^N \delta_{\theta^i}, 
    %~~ \{\xi^i\}_{i=1}^N\sim \mathcal N(0,I)
\end{equation*}
where the support points $\theta^i$ are treated as optimization variables. The gradient of the objective with respect to each variable  $\theta^i$ is given by $\phi(0,\theta^i)$, where $\phi$ solves the PDE~\eqref{eq:pde}. 

We approximate $\phi$ using both TR-BSDE (Algorithm~\ref{alg:phi}) and PNAA (Algorithm~\ref{alg:pnaa}) , and use the resulting gradients to optimize  for $\theta^i$ through a gradient descent procedure. For comparison, we also implement a standard automatic differentiation (AutoDiff) baseline, which computes gradients by backpropagating through a time-discretized version of the SDE. We run the optimization algorithm for $15000$ iterations. 
The resulting accumulation points of $\theta^i$ are shown in Fig \ref{fig:IP-heat-ad} and Fig.~\ref{fig:IP-heat-pnaa}, together with a heat map of the objective function $\Expect[J(\xi)\mid \xi=\theta]$, estimated by Monte Carlo simulation.
It is observed that the points produced by TR-BSDE concentrate more tightly in low-cost regions than those obtained by PNAA. Moreover, the PNAA result includes several points that drift far from the origin and are therefore omitted from the figure.  Compared with the AutoDiff baseline, TR-BSDE also converges substantially faster.  Overall, these results indicate that the proposed method yields more accurate and lower-variance gradient estimates for this nonlinear system. 
% It is observed that our proposed TR-BSDE algorithm provides a more accurate and less variance gradient estimation for the nonlinear system. The automatic differentiation method converges much more slowly than our method, and the $\phi$ learned from non-adapted adjoint process has some bias in estimation.

\subsection{Fine-tuning diffusion model}\label{sec:num-diff}
We consider a pretrained one-dimensional diffusion model
that generates samples from the bimodal distribution 
% , which generates samples in a Bimodal distribution, governed by the following SDE
% \begin{equation}
%     \ud X_t = (aX_t  + \sigma^2 s(t,X_t)) \ud t + \sigma \ud W_t, ~~X_0\sim p_0,~X_T \sim p_T,
% \end{equation}
% where $a=2$, $\sigma=2$, $T=1$, $p_0 = \mathcal N(0,1)$, time step $\Delta t =0.02$, and 
\[p_T = 0.5\mathcal{N}(-3,1) + 0.5\mathcal N(3,1)\]
% \begin{equation*}
%     p_T = 0.5\mathcal{N}(-3,1) + 0.5\mathcal N(3,1)
% \end{equation*}
Our goal is to fine-tune this model so that it generates samples from the tilted distribution defined through the terminal cost
\begin{equation*}
    \ell_f(x) = \beta \frac{1}{2}(x-3)^2
\end{equation*}
where $\beta>0$ is the tilting parameter. We assume the nominal initial distribution $p_0 = \mathcal N(0,1)$, and parameterize the modified initial distribution  $q_0=\mathcal N(\mu,Q^2)$, where $\mu$ and $Q$ are optimization variables. The objective is to solve the SOC problem~\eqref{eq:SOC} by jointly optimizing the initial distribution  $q_0$ and the control law $U_t=k(t,X_t)$. 

\medskip
\noindent
{\bf Initial distribution optimization:} To optimize the initial distribution, we use the closed-form expression for the KL divergence between one-dimensional Gaussian distributions 
\begin{equation*}
    D_\mathrm{KL}(q_0\|p_0) = \frac{1}{2}(Q^2 + \mu^2 - \log Q^2 -1),
\end{equation*}
then we can obtain the the gradient of cost with respect to the initial distribution mean and variance
\begin{subequations}
\label{eq:ini-grad}
\begin{align}
    \frac{\partial}{\partial \mu} [\Expect_{\xi \sim q_0}[ J(\xi,U)]  +  D_\mathrm{KL}(q_0\|p_0)] = & \Expect[Y_0] + \mu,\\ 
    \frac{\partial}{\partial Q}[\Expect_{\xi \sim q_0}[ J(\xi,U)]  +  D_\mathrm{KL}(q_0\|p_0)]  = &\Expect[Y_0(X_0-\mu)/Q] \nonumber \\&+  Q - 1/Q.
\end{align}
\end{subequations}

\medskip
\noindent
{\bf TR-BSDE:} To compute the optimal control law, we initialize with the zero control $k(t,x)=0$, solve the BSDE, and learn $\phi$ from Algorithm~\ref{alg:phi}. After each update of $\phi$, we set
 $k(t,x) = - g(t)^\top \phi(t,x)$. This choice is motivated by the stochastic maximum principle, which implies that at optimality  $U_t=-g(t)^\top Y_t$ and $Y_t=\phi(t,X_t)$. The parameters of the initial Gaussian distribution are updated using~\eqref{eq:ini-grad} and $Y_0 = \phi(0,x)$. The procedure appears in Algorithm~\ref{alg:fine-tune}. 
% \begin{equation*}
%     \partial_\mu J = \beta \Expect[Y_0] + \mu,~~~\partial_Q J =\beta \Expect[Y_0(X_0-\mu)/Q] +  Q - 1/Q
% \end{equation*}

% We can use the KL divergence formula between Gaussian distributions to compute the derivative of the KL-divergence term with respect to $\mu$ and $Q$:  
% \begin{equation*}
%     J(q_0,U_t) = \Expect\left[\beta\ell_f(X_T) + \int_0^T \frac{1}{2} \|U_t\|^2 \ud t \right] + D_\mathrm{KL}(q_0\|p_0)
% \end{equation*}
% We defined terminal cost as $\ell_f(x) = \frac{1}{2}(x-3)^2$, and we assume the optimal initial distribution is Gaussian distribution with mean $\mu$ and standard derivation $Q$, i.e., $q_0 = \mathcal N(\mu,Q^2)$. Under this setting, the KL divergence between $q_0$ and standard Gaussian distribution has an explicit expression
% \begin{equation*}
%     D_\mathrm{KL}(q_0\|p_0) = \frac{1}{2}(Q^2 + \mu^2 - \log Q^2 -1),
% \end{equation*}
% then we can obtain the the gradient of cost with respect to the initial distribution mean and variance
% In our method, after learning $\phi(t,\cdot)$ between $\Xr_t$ and $\Yr_t$, The control law is updated through $U_t = -\beta \sigma \phi(t,\cdot)$. 

\medskip
\noindent 
{\bf Adjoint matching~\cite{domingoadjoint}:} The control law $k(t,x)$ is parameterized by a neural network and learned by regressing $k(t,X_t)$ onto $-g(t)^\top \Yn_t$, where $\Yn_t$ is the non-adapted adjoint process. After a fixed number of regression iterations, the adjoint process is re-simulated under the updated control law. The procedure appears in Algorithm \ref{alg:ad-mat}. As in TR-BSDE, the parameters of the initial Gaussian distribution are updated using~\eqref{eq:ini-grad}, with the difference that $Y_0$ is replaced by the non-adapted process $\Yn_0$.  

% \begin{equation*}
%     U_t = -\beta \sigma \phi(t,\cdot)
% \end{equation*}
We compare two methods for tilting parameter $\beta \in \{\frac{1}{50}, \frac{1}{20}, \frac{1}{8}, \frac{1}{6}, 1\}$.
% with all other same setting. 
 % The experimental settings are as follows. The fine-tuning iteration number is set to $k_f=30$. The number of iterations for the time-reversal BSDE solver is $J_f=5$ when $\beta \in \{1/50,1/20,1/8\}$ and $J_f=10$ otherwise. The optimization iteration number is $1000$, and the sample size is $N=10000$. The score function and $\phi$ are parameterized by two neural networks. For $\beta=1/50$, we initialize the control input as zero; for $\beta > 1/50$, we use the training result from the smaller $\beta$ as a warm start.
The experimental results are shown in Fig.~\ref{fig:sample-compare}. It is observed that when $\beta$ is small, the two methods exhibit similar performance. However, as $\beta$ increases, the fine-tuning task becomes more challenging, and the proposed TR-BSDE outperforms the Adjoint matching approach.

\subsection{Double-well benchmark}

We evaluate TR-BSDE and adjoint matching on the double-well benchmark 
from~\cite{nusken2021solving} as the problem dimension increases. The model 
is governed by the SDE
\begin{equation}
    \ud X_t = -\nabla\Psi(X_t)\,\ud t + U_t\,\ud t + \ud W_t, \qquad X_0 = 0,
\end{equation}
with cost
\begin{equation}
    J = \mathbb{E}\!\left[\int_0^T\frac{1}{2}\|U_s\|^2\,\ud s 
    + \ell_f(X_T)\right],
\end{equation}
where
\[
    \Psi(x) = \sum_{i=1}^n \kappa_i(x_i^2-1)^2, \qquad 
    \ell_f(x) = \sum_{i=1}^n v_i(x_i^2-1)^2.
\]
We set $\kappa_i = v_i = 1$, $T=1$, $\Delta t = 0.005$, and vary the 
dimension $n \in \{5,6,7,8,9,10\}$. Performance is measured by the relative 
cost gap $(J-J^\star)/J^\star$, estimated with $N=10000$ samples, where 
$J^\star$ is the optimal cost. The uncontrolled system is included as a 
baseline. As shown in Fig.~\ref{fig:double-well}, TR-BSDE consistently 
outperforms adjoint matching across all dimensions, with the relative cost 
gap remaining small and stable as $n$ increases

\begin{figure}[h]
    \centering
    \includegraphics[width=0.75\linewidth]{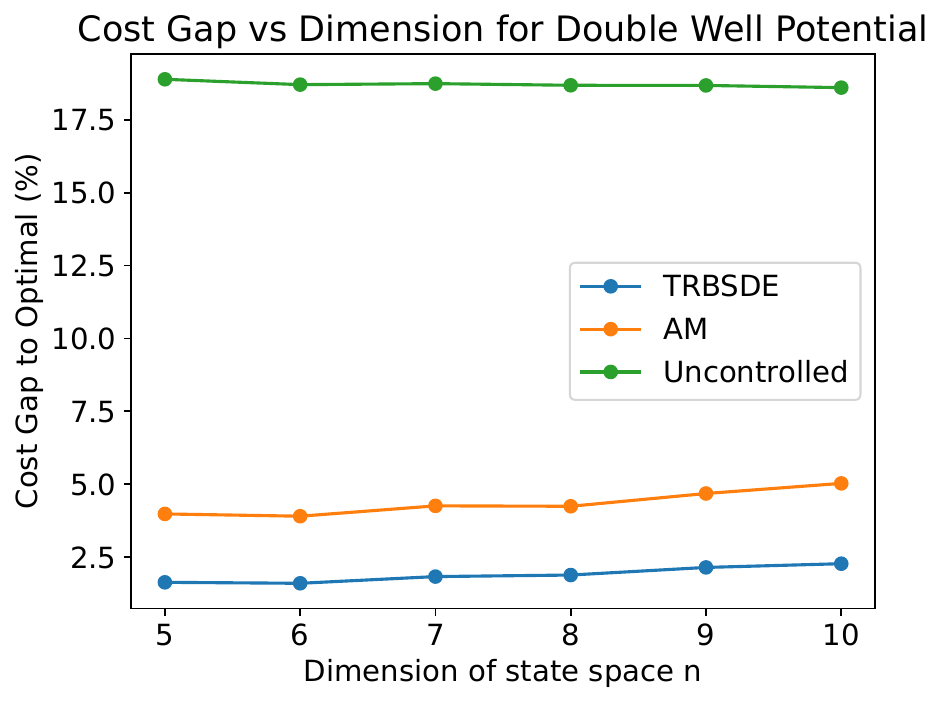}
    \caption{Comparison of TR-BSDE and Adjoint Matching on the double well benchmark across state dimensions $n \in \{5, 6, 7, 8,9, 10\}$. The y-axis reports the relative cost gap $(J - J^\star)/J^\star$, where $J$ is the realized cost under the learned controller and $J^\star$ is the optimal cost obtained from the path integral representation. Each estimate uses $N = 10000$ Monte Carlo samples. TRBSDE consistently achieves a smaller gap than Adjoint Matching as $n$ increases.}
    \label{fig:double-well}
\end{figure}

\begin{remark}[Computational limitations]
The main computational bottleneck of TR-BSDE, relative to adjoint matching, 
is the trace term in~\eqref{eq:tr-bsde}, which requires evaluating the 
Hessian of a neural network and becomes increasingly expensive as $n$ grows. 
Adjoint matching avoids this cost but incurs higher gradient variance, as 
demonstrated both theoretically in Section~\ref{sec:adjoint} and numerically. A natural remedy is to 
replace the exact trace with Hutchinson's estimator~\cite{hutchinson1989stochastic}, 
a stochastic approximation that has proven effective in similar 
settings~\cite{song2020sliced}, and which we leave 
for future work. Additionally, TR-BSDE requires an estimated score function, 
which is an extra ingredient compared to adjoint matching.
\end{remark}

\bibliographystyle{IEEEtran}
\bibliography{Ref}

\end{document}

%% file: _definitions.tex
\def\Re{\mathbb{R}}
\def\R{\mathbb{R}}

\def\argmin{\mathop{\text{\rm arg\,min}}}

\def\notes#1{\marginpar{\tiny #1}\typeout{Notes!
Notes!
Notes!
}}
\renewcommand{\notes}[1]{\typeout{notes!}}

\def\Re{\field{R}}

\def\Expect{{\mathbb E}}

\def\R{\mathbb{R}}

\newtheorem{remark}{Remark}
\newtheorem{proposition}{Proposition}

\def\beq{\begin{eqnarray}} 
\def\bc{\begin{center}} 
\def\be{\begin{enumerate}}
\def\bi{\begin{itemize}} 
\def\bs{\begin{small}}
\def\bS{\begin{slide}}
\def\ec{\end{center}} 
\def\ee{\end{enumerate}}
\def\ei{\end{itemize}}
\def\es{\end{small}}
\def\eS{\end{slide}}
\def\eeq{\end{eqnarray}}

\newcommand{\trace}{\text{Tr}}

\newcommand{\ud}{\,\mathrm{d}}

\def\Re{\mathbb{R}}

\def\argmin{\mathop{\text{\rm arg\,min}}}

\renewcommand{\Re}{\mathbb{R}}

\newcounter{rmnum}

\newcounter{anum}